\documentclass[12pt]{article}
\usepackage[dvips]{color}
\usepackage{authblk}

\usepackage[all]{xy}
\usepackage{amsmath,amssymb,amsfonts,amsthm}
\usepackage{eucal}
\usepackage{verbatim} 
\usepackage{graphicx} 

\usepackage{lineno}
\usepackage{ifthen} 

\usepackage{stmaryrd}

\newtheorem{theorem}{Theorem}[section]

\newtheorem{corollary}[theorem]{Corollary}

\newtheorem{example}[theorem]{Example}

\newtheorem{proposition}[theorem]{Proposition}

\theoremstyle{definition}
\newtheorem{definition}[theorem]{Definition}
\theoremstyle{remark}

\numberwithin{equation}{section}

\begin{document}


\title{Global stability in some one-dimensional non-autonomous discrete periodic population models}

\author[1,2]{Rafael Lu\'{i}s\thanks{rafael.luis.madeira@gmail.com}}
\author[1]{Elias Rodrigues \thanks{elias@uma.pt}}
\affil[1]{University of Madeira, Funchal, Madeira, Portugal.}
\affil[2]{Center for Mathematical Analysis, Geometry, and Dynamical Systems,
University of Lisbon, Lisbon, Portugal.}

\maketitle

\begin{abstract}
 For some one-dimensional discrete-time autonomous population models, local stability implies global stability of the positive equilibrismo point. One of the known techniques is the enveloping method. In this paper we extend the enveloping method to one single periodic population models. We show that, under certain conditions,  ``individual enveloping" implies  ``periodic enveloping" in one-dimensional periodic population models.  
\end{abstract}

\textbf{Keys Words:} Periodic population models, Local stability, Global stability, Enveloping, Applications.

\textbf{AMS:} {Primary 37C75},  Secondary{ 39A23, 39A30, 39A60.}

\section{Introduction}

One-dimensional models are an appropriate mathematical tool to model the behavior of populations with  non-overlapping generations. This subject has been intensely investigated by different researchers.

An \textbf{autonomous population model} is a difference equation of the form $$x_n=f(x_n),~~  n \in \mathbb{Z}^{+}_0,$$ where the map $f$ is a continuous function from the nonnegative reals to the nonnegative reals and there is a unique positive number $x^*$, the equilibrium point, such that 

\begin{eqnarray*}
f(0)  &=&0, \\
f(x) &>&x\text{ for }0<x<x^*, \\
f(x) &=&x\text{ for }x=x^*,\\
f(x) &<&x\text{ for }x>x^*.
\end{eqnarray*}
Furthermore, the map $f$ is bounded on $[0,x^*]$. Notice that these conditions correspond to the usual assumptions in population dynamics.

After normalization, we can always assume that $x^*=1$. In \cite{ElaydiDC} we can find a complete study for the local properties of $x^*$. In general, it is much more complicated to investigate the global stability of $x^*$. The following result found in  \cite[Corollary 2.4]{SERS2004} and \cite[Theorem 2.1, pp 47]{SMR1993} gives  a condition on global stability.
\begin{theorem}\label{tSh}
Let $x^*$ be a fixed point of a continuous map $f$ on the compact interval $[a, b]$. Then $x^*$ is globally asymptotically stable relative  to the interval $(a, b)$ if and only if $f^2(x) > x$ for $x < x^*$ and $f^2(x) < x$ for $x > x^*$, for all $x\in (a, b)$.
\end{theorem}
In other words, a continuous population model is globally stable if and only if it has no cycles of minimal period 2. This result was noticed much earlier by Coppel in 1955 \cite{Coppel1955}: 

\begin{theorem}\label{coppel}
\label{thm_no_periodic_points} Let $I=[a, b]\subseteq\mathbb{R}$ and $f:I
\rightarrow I$ be a continuous map. If $f$ has no points of prime period
two, then every orbit under the map $f$ converges to a fixed point.
\end{theorem}

Unfortunately, this global stability condition may be difficult to test. Moreover, it seems that there is no obvious connection between the local stability conditions  and the global stability conditions.

In a series of papers \cite{Cull2003, Cull2005, Cull2006,CWW2008} P. Call and his collaborators embarked in the theory that ``enveloping" implies global stability. A function $h(x)$ envelops a function $f(x)$ if and only if
\begin{enumerate}
\item [(i)] $h(x)>f(x)$ for all $x\in (0,x^*=1)$;
\item [(ii)] $h(x)<f(x)$ for $x>x^*=1$ such that $h(x)>0$ and $f(x)>0$.
\end{enumerate}
Combining this definition with Theorem \ref{tSh} and Theorem \ref{coppel}, one has the following result.
\begin{theorem}\label{GlobalSt}(\cite{CWW2008})
If $f(x)$ is enveloped by $g(x)$, and $g(x)$ is globally stable, then $f(x)$ is
globally stable.
\end{theorem}

Hence, the enveloping function plays a central rule in this theory. In \cite{Cull2006} P. Cull presented the following result concerning the enveloping function.
\begin{theorem}\label{thmonotone}
Let $h(x)$ be a monotone decreasing function which is positive on $(0,x_{h}>1)$ and so that $h(h(x))=x$. Assume that $f(x)$ is a continuous function such that
\begin{eqnarray*}
h(x)  &>&f(x)\text{ on }(0,1), \\
h(x) &<&f(x)\text{ on }(1,x_h), \\
f(x) &>&x\text{ on }(0,1),\\
f(x) &<&x\text{ on }(1,\infty)\\
f(x)&>&0\text{ whenever }x>1.
\end{eqnarray*} 
Then for all $x>0$, $\lim_{n\rightarrow\infty}f^n(x)=1$.
\end{theorem}
The preceding theorem shows the importance of the enveloping function in global stability. So, the challenge will be to find the apropriar enveloping. Surprisingly  or not, the following M\"{o}bius transformation may help in finding the appropriate enveloping:
\begin{theorem}\cite{Cull2006}
If $f(x)$ is enveloped by a linear fractional function of the form $h(x)=\frac{1-\alpha x}{\alpha-(2\alpha-1)x}$, $\alpha\in[0,1)$, then $f(x)$ is globally stable.
\label{CORL_fractional}
\end{theorem}

Since there are many options for enveloping in the preceding Theorem, it is necessary to adjust each particular model with the enveloping in the set of parameters. Hence, this task is not easy.

Eduardo Liz \cite{Liz2007} tried to simplify this process. He studied enveloping function for  models of the form $x_{n+1}=x_n+f(x_n)$ and uses Schwarzian derivative of $f$. He was able to characterize when Local Stability implies Global Stability in certain one-dimensional population models. His results are based in the following proposition:
\begin{proposition}\cite{Liz2007}
Let $f$ to be a population model defined as before and suppose that $f$ is a $C^3$ map that has at most one critical point $x_c$. If $|f^\prime(x^*)|\leq 1$ and $Sf(x)<0$ for all $x\neq x_c$, then $x^*$ is a globally stable fixed point of $f$, where $Sf$ is the Schwarzian derivative of $f$ given by
\[
Sf(x)=\dfrac{f^{\prime\prime\prime}(x)}{f^{\prime}(x)}-\dfrac{3}{2}\left(\dfrac{f^{\prime\prime}(x)}{f^{\prime}(x)}\right)^2.
\]
\end{proposition}

Later on, Rubi\'{o}-Masseg\'{u} and Ma\~{n}osa in \cite{manosa2007} brings up the importance of the enveloping function since the enveloping implies the existence of a Global Lyapunov function. Thus the global asymptotically stability can be seen as a consequence of an invariant principle.

Our main objective in this paper is to extend the theory of enveloping to one-dimensional periodic population models.

In Section \ref{ps} we present some preliminaries concerning the theory of non-autonomous periodic difference equations. In the next section we present our main results, i.e., we give the condition for which individual enveloping implies periodic enveloping for mappings. The next section is devoted to applications. We illustrate our results in several well known models in population dynamics as is the cases of  the periodic Ricker model, the periodic generalized Beverton-Holt model and the periodic logistic model. A mixing population models are studied as well.

\section{Periodic systems}\label{ps}
A difference equation is called non-autonomous if it is governed by the rule
\begin{equation}
x_{n+1}= f_{n}(x_{n}), n \in \mathbb{Z}^{+},  \label{nade}
\end{equation}
where $x\in X$ and $X$ is a topological space. Here the orbit of a point $x_0$ is generated by the composition of the sequence of maps 
\begin{equation*}
f_0,f_1,f_2,\ldots.
\end{equation*}
Explicitly, 
\begin{eqnarray*}
x_1&=&f_0(x_0), \\
x_2&=&f_1(x_1)=f_1\circ f_0(x_0), \\
&&\vdots \\
x_{n+1}&=&f_{n}\circ f_{n-1}\circ\ldots\circ f_1\circ f_0(x_0), \\
&&\vdots
\end{eqnarray*}

If the sequence of maps is periodic, i.e., $f_{n+p}=f_n$, for all $%
n=0,1,2\ldots$ and some positive integer $p>1$, then we talk about
non-autonomous periodic difference equations. Systems where the sequence of
maps is periodic, model population with fluctuation habitat, and they are
commonly called periodically forced systems. 

Throughout  this paper we work with non-autonomous periodic difference equation in which $p$ is the minimal period of Equation  (\ref{nade}) and $X=\mathbb{R}$.

Notice that the non-autonomous periodic difference equation (\ref{nade})
does not generate a discrete (semi)dynamical system \cite{ESUAE} as it may
not satisfy the (semi)group property. One of the most effective ways of
converting the non-autonomous difference equation (\ref{nade}) into a
genuine discrete (semi)dynamical system is the construction of the
associated skew-product system as described in a series of papers by Elaydi
and Sacker \cite{ElSa2005, ElSa2005a, ESUAE, ES2006}. It is noteworthy to
mention that this idea was originally used to study non-autonomous
differential equations by Sacker and Sell \cite{SS1977}.

\begin{definition}
An ordered set of points $C_r=\left\{%
\overline{x}_{0},\overline{x}_{1},\ldots,\overline{x}_{r-1}\right\}$ is an $
r-$periodic cycle in $X$ if 
\begin{equation*}
f_{(i+nr)\bmod p}(\overline{x}_{i}) = \overline{x}_{(i+1)\bmod r} , n\in\mathbb{Z}^+.
\end{equation*}
In particular,
\[f_i(\overline{x}_{i}) = \overline{x}_{i+1}, 0\leq i\leq r-2,\]
and
\[f_t(\overline{x}_{t\bmod r}) = \overline{x}_{(t+1)\bmod  {r}}, 
r-1 \leq t \leq p-1.\]
\label{dpc}
\end{definition}

It should be noted that the $r-$periodic cycle $C_{r}$ in $X$ generates an $
s-$periodic cycle on the skew-product $X \times Y$ ($Y=\{f_0,f_1,\ldots,f_{p-1}\}$) of the form 
\begin{equation*}
\widehat{C}_{s}= \{(\overline{x}_0,f_0),(\overline{x}_1,f_1),...,(\overline{x
}_{(s-1)\bmod r},f_{(s-1)\bmod p})\},
\end{equation*}
where $s = lcm[r, p]$ is the least common multiple of $r$ and $p$.

To distinguish these two cycles, the $r-$periodic cycle $C_{r}$ on $X$ is
called an $r-$geometric cycle (or simply $r-$periodic cycle when there is no
confusion), and the $s-$periodic cycle $
\widehat{C}_s$ on $X\times Y$ is called an $s-$complete cycle. Notice that either $r<p$, or $r=p$ or $r>p$.

Define the composition operator $\Phi$ as follows
\[
\Phi_n^i=f_{n+i-1}\circ\ldots\circ f_{i+1}\circ f_i.
\]
When $i=0$ we write $\Phi_n^0$ as $\Phi_n$.

As a consequence of the above remarks it follows that the $s-$complete cycle $\widehat{C}_s$ is a fixed point of the composition operator $\Phi_s^i$. In other words we have that 
\[\Phi_s^i(\overline{x}_{i \bmod  {r}})=\overline{x}_{i \bmod  {r}}.\] 
If the sequence of maps $\{f_i\},i\geq 0$ is a parameter family of maps one-to-one in the parameter, then by \cite{ELOspringer2011} we have that $\overline{x}_{i \bmod  {p}}$ is a fixed point of $\Phi_p$.

\section{Enveloping in periodic models}
In this section we extend the idea of enveloping to periodic single species population models.

In population dynamics it is common to work with a parameter family of maps. If we are working with a certain population model, we can always rescale the positive fixed point to $x^*=1$. Hence, under this scenario, a parameter family of maps will have the same positive fixed point. Having in mind this idea, we point out the first  assumption of our work.

\textit{\textbf{H1 -} Let $\mathcal{F}=\left\{f_0,f_1,f_2,\ldots\right\}$ to be a set of $C^1$ population models such that $f_i(1)=1$, for all $i=0,1,2,\ldots$. Assume that the composition 
\[
\Phi_p(x)=f_{p-1}\circ\ldots\circ f_1\circ f_0(x)
\]
is continuous in a subset of the nonnegative reals. Further, in order to guaranty periodicity of the equation $$x_{n+1}=f_n(x_n),$$ we also require that the maps on $\mathcal{F}$ are periodic with period $p$, i.e., $f_{n+p}=f_n$, for all $n$.
}

A natural question arises in the field of periodic difference equations: is the composition of population models a population model?

\begin{figure}
\centering
\includegraphics[scale=0.25]{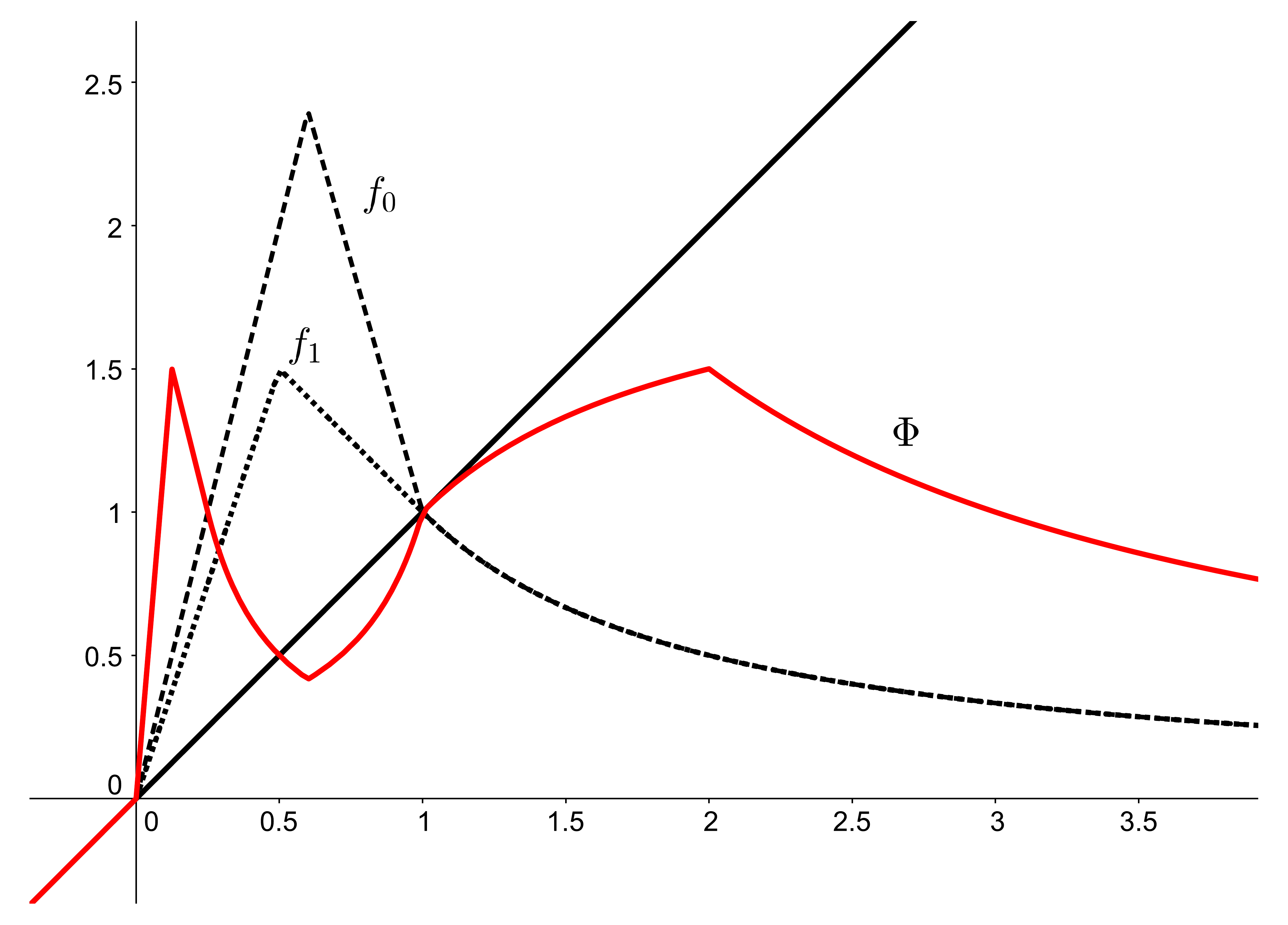}
\caption{This example shows that the composition of population models may not be a population model}
\label{fig:picewise}
\end{figure}

The answer of this question, in general, is negative, it depends on the individual maps. In Figure \ref{fig:picewise} we present a concrete example, where
\[
f_{0}(x)=\left\{ 
\begin{array}{l}
4x\text{ if }0\leq x<0.6 \\ [5pt]
-3.5x+4.5\text{ if }0.6\leq x<1 \\ [5pt]
\dfrac{1}{x}\text{ if }x>1
\end{array}
\right. 
\]
and
\[
f_{1}(x)=\left\{ 
\begin{array}{l}
3x\text{ if }0\leq x<0.5 \\ [5pt]
-x+2\text{ if }0.5\leq x<1 \\ [5pt]
\dfrac{1}{x}\text{ if }x>1
\end{array}
\right. .
\]

Observe that the maps of $\mathcal{F}$  are increasing in certain interval. Hence, one can show the following proposition:

\begin{proposition}
Under hypothesis \textbf{H1} the composition map $\Phi_p(x)$ is increasing in  $(0,c_\Phi)$, for certain positive value $c_\Phi$. Moreover, there exists $x_\Phi<x^*_\Phi$ such that $\Phi_p(x)>f_i(x)$,  for all $x\in (0,x_\Phi)$, $i\in\{0,1,\dots, p-1\}$ with $\Phi_p(x^*_\Phi)=x^*_\Phi$. 
\label{Linc}
\end{proposition}
\begin{proof}

It follows from our hypothesis that each individual map $f_i$ is increasing on $(0,c_i)$, for some $c_i>0$ (eventually the map $f_i$ can be increasing in all the domain). Since the composition of monotone mappings is a monotone map,  it follows that there exists $c_\Phi>0$, the minimum of the critical values of $\Phi_p$,  such that $\Phi_p(x)$ is an increasing function on $(0,c_\Phi)$.

Since each one of the individual maps $f_i$ is increasing on $(0,c_i)$ and the origin is an unstable fixed point of all the maps, we have that $f_i^\prime(0)>1$, for all $i\in\{0,1,\dots, p-1\}$. Taking the derivative of the composition map we have 
\begin{eqnarray*}
\Phi_p^\prime(x)&=&f^\prime_{p-1}(\Phi_{p-1}(x))\times f^\prime_{p-2}(\Phi_{p-2}(x))\times\ldots f_1^\prime(f_0(x))\times f^\prime_0(x)\\
&=&\prod_{i=0}^{p-1}f^\prime_i(\Phi_i(x)).
\end{eqnarray*}
Hence, at the origin we have 
\[
\Phi_p^\prime(0)=\prod_{i=0}^{p-1}f^\prime_i(\Phi_i(0))>> f^\prime_i(0), i\in\{0,1,\dots, p-1\}.
\]
This implies that there exists a positive number $x_\Phi<x_\Phi^*$ such that 
\[
\Phi_p(x)>f_i(x),\text{~~for all~~}   x\in (0,x_\Phi), ~~i\in\{0,1,\dots, p-1\}.
\]

\end{proof}
\bigskip

The next assumption will be in the enveloping function. First, let us observe the following example where we show that ``\textbf{individual enveloping}" do not implies ``\textbf{periodic enveloping}".
\begin{example}
Let $f_0(x)=xe^{1.5(1-x)}$ and $f_1(x)=xe^{1.2(1-x)}$ and assume that $f_{n+2}=f_{n}$, for all $n$. It is clear that $f_0,f_1\in\mathcal{F}$ and $x_{n+1}=f_{n}(x_n)$ is a $2-$periodic difference equation. 

Each one of the maps $f_0$ and $f_1$ is enveloped by the map $g(x)=xe^{2(1-x)}$. The map $g(x)$ is a globally asymptotically stable  population model since it is enveloped by the decreasing fractional function $h(x)=2-x$ (for more details see \cite{CWW2008}). Moreover, from Theorem \ref{GlobalSt} one can conclude that the individual maps $f_i(x)$, $i=0,1$ are globally stable since they are envelop by $g(x)$.

To study the dynamics of the $2-$periodic difference equation we study the dynamics of the map
\[
\Phi_2(x)=f_1\circ f_0(x)=x e^{2.7-1.5x-1.2 x e^{1.5(1-x)}}.
\]
Plotting the graph of $\Phi_2(x)$ one can conclude that there exists a positive value $a<1=x^*_{\Phi}$ such that 
\[
\Phi_2(x)>xe^{2(1-x)}, \text{ for all }x\in(0,a)
\]
and 
\[
\Phi_2(x)<xe^{2(1-x)}, \text{ for all }x\in(a,1).
\]
Consequently, $\Phi_2(x)$ is not enveloped by $g(x)$. However, the individual maps $f_0$ and $f_1$ are enveloped by $g(x)$.  Hence, we can not conclude stability of $\Phi_2(x)$ from the individual enveloping $g(x)$. In Section \ref{grm} we will show that $\Phi_2(x)$ is also enveloped by $h(x)=2-x$ and consequently from Theorem \ref{CORL_fractional}  it is globally stable. 
\end{example}

Our goal in this work is to find certain class of maps where \underline{\textbf{individual enveloping}} implies \underline{\textbf{periodic enveloping}}. Due the preceding examples we have to guarantee: (i) the composition of population models is a population model and (ii) the individual enveloping is also an enveloping for the  composition map, i.e., we have to guarantee that the composition map has two fixed points, the origin and a positive fixed point $x_{\Phi}^*$, $\Phi_p(x)>x$ if $x\in(0,x^*_\Phi)$ and $\Phi(x)_p<x$ if $x>x^*_\Phi$, and there exits an enveloping for a sequence of the individual maps $f_i$ that envelops the composition map. These observations motivates the second assumption:

\textit{\textbf{H2 -} There exits a decreasing envelop $h$ such that $h(x)$ envelops all the maps in the set $\mathcal{F}$ and $h\circ h(x)=x$.}

The following proposition may help us in the construction of an envelop. We will omit the prove since the result follows directly by symmetry.

\begin{proposition}\label{PENV}
Let $h$ to be an enveloping of a population model $f$  in the conditions of hypothesis \textbf{H2}.
Consider the graph of $f$ and the curve, 
$S_f$, obtained from the graph of $f$ by symmetry with respect to the diagonal $y=x$. Then, the graph of $h$ lies  between the graph of $f$ and the curve $S_f$ everywhere, with the exception of the fixed point $x^*=1$.
\end{proposition}

We are now ready to present our main result of this section. It states the conditions in which individual enveloping implies periodic enveloping.
\begin{theorem}\label{GS}
Under hypothesis \textbf{H1} and \textbf{H2}, the composition map $\Phi_p(x)$ is a globally asymptotically stable population model.
\end{theorem}
\begin{proof}
The prove follows by induction. We will show the result for the composition of two maps ($p=2$) and omit the general case ($p>2$).

Combining \textbf{H1} and \textbf{H2} it follows from Theorem \ref{thmonotone} that $f_i(x)$, $i=0,1,2,\ldots$ is a sequence of globally asymptotically stable population models. 

Since $f_i(0)=0$ and $f_i(1)=1$, for all $i=0,1,2,\ldots$ it follows that $\Phi_p(0)=0$ and $\Phi_p(1)=1$. The uniqueness of $x^*_{\Phi}=1$ will follow in the following arguments.

Let us first study the composition map $\Phi_2(x)=f_1\circ f_0(x)$. Since $\Phi_2(1)=1$ we split the prove into two cases: \textbf{I} - $0<x<1$ and \textbf{II} - $x>1$.

\textbf{Case I:} If $0<x<1$, either $f_0(x)>1$ or $x<f_0(x)<1$ (recall that $f_0(x)>x$ when $x\in(0,1)$). 

If $f_0(x)>1$ then $f_1\circ f_0(x)=f_1\left(f_0(x)\right)<f_0(x)<h(x)$. Since $h$ is decreasing and $f_1(y)>h(y)$, for all $y>1$, we have that 
\[
x=h\circ h(x)<h\circ f_0(x)<f_1\circ f_0 (x).
\]
If $x<f_0(x)<1$, then $h\circ f_0(x)<h(x)$. But, since $f_0(x)<1$ we have $f_1\circ f_0 (x)<h\circ f_0(x)<h(x)$. On  the other hand $x<f_0(x)<f_1\circ f_0(x)$.

\textbf{Case II:} Let $x>1$ and assume first $0<f_0(x)<1$. This implies that $f_1\circ f_0(x)>f_0(x)>h(x)$. In order to prove that $f_1\circ f_0(x)<x$ we notice that 
\[
x=h\circ h(x)>h\circ f_0(x)>f_1\circ f_0 (x).
\]
When $x>f_0(x)>1$, we have that $h\circ f_0(x)>h(x)$ and consequently $f_1\circ f_0(x)>h\circ f_0(x)>h(x)$. Since $f_0(x)>1$ it follows that $x>f_0(x)>f_1\circ f_0(x)$.

We have shown that $\Phi_2(x)$ is a population model and the enveloping function $h$ envelops the composition map $\Phi_2$. Consequently, from Theorem \ref{thmonotone} the map $\Phi_2(x)$ is a globally asymptotically stable population model.

\end{proof}
It will be beneficial, in certain cases, to have the contrapositive of this theorem which we write in the following corollary: 
\begin{corollary}\label{CGS}
If $f_0,f_1,\ldots, f_{p-1}$  are population models and the composition map $\Phi_p$ is not a population model, then it can not exist an enveloping $h$, in the conditions of hypotheses \textbf{H2}, such that $h$ envelops all the individual population models.
\end{corollary}

\section{Applications}
In this section we illustrate our results in  some one-dimensional periodic population models. We study a family of Ricker maps, a family of Beverton-Holt models, quadratic models, mixing models and harvesting models.

\subsection{Ricker model}\label{grm}
Let us consider the periodic difference equation given by the following equation 
\[
x_{n+1}=R_n(x_n),
\] 
where the sequence of maps $R_n(x)$ is given by
\begin{equation}
R_n(x)=x e^{r_n(1-x)},~~r_n>0, ~~n=0,1,2\ldots.
\end{equation}

The local stability condition $0<r_n\leq 2$, $n\in\{0,1,\ldots\}$ of each individual population model $R_i(x)$ implies global stability of $R_i(x)$, since each one of the individual maps  $R_n(x)$, $n\in\{0,1,\ldots\}$ is enveloped by $h(x)=2-x$, which is a fractional decreasing function with $\alpha=1/2$ (See Theorem \ref{CORL_fractional}). In other words $x^*=1$ is a globally asymptotically stable fixed point of $R_n(x)$, $n\in\{0,1,\ldots\}$.  Notice that $h\circ h(x)=x$. Thus, hypothesis \textbf{H2} is satisfied.

In order to have periodicity we require that  $R_{n+p}=R_{n}$, for all $n=0,1,2,\ldots$, i.e., the sequence of parameters satisfies $r_n=r_{n\bmod p}$ for all $n$. It is clear that the composition map 
$$\Phi_p(x)=R_{p-1}\circ\ldots\circ R_1\circ R_0(x)$$ is continuous in $\mathbb{R}^+_0$.
Consequently, hypothesis \textbf{H1} is satisfied. Hence, from Theorem \ref{GS} it follows that $\Phi_p(x)$ is a globally asymptotically stable population model, i.e.,  the $p-$periodic Ricker difference equation is globally stable whenever $r_n\in(0,2]$, $n=0,1,2,\ldots$.

Hence, in this family of population models, individual enveloping implies periodic enveloping.

\begin{figure}[t]
\centering
\includegraphics[scale=0.5]{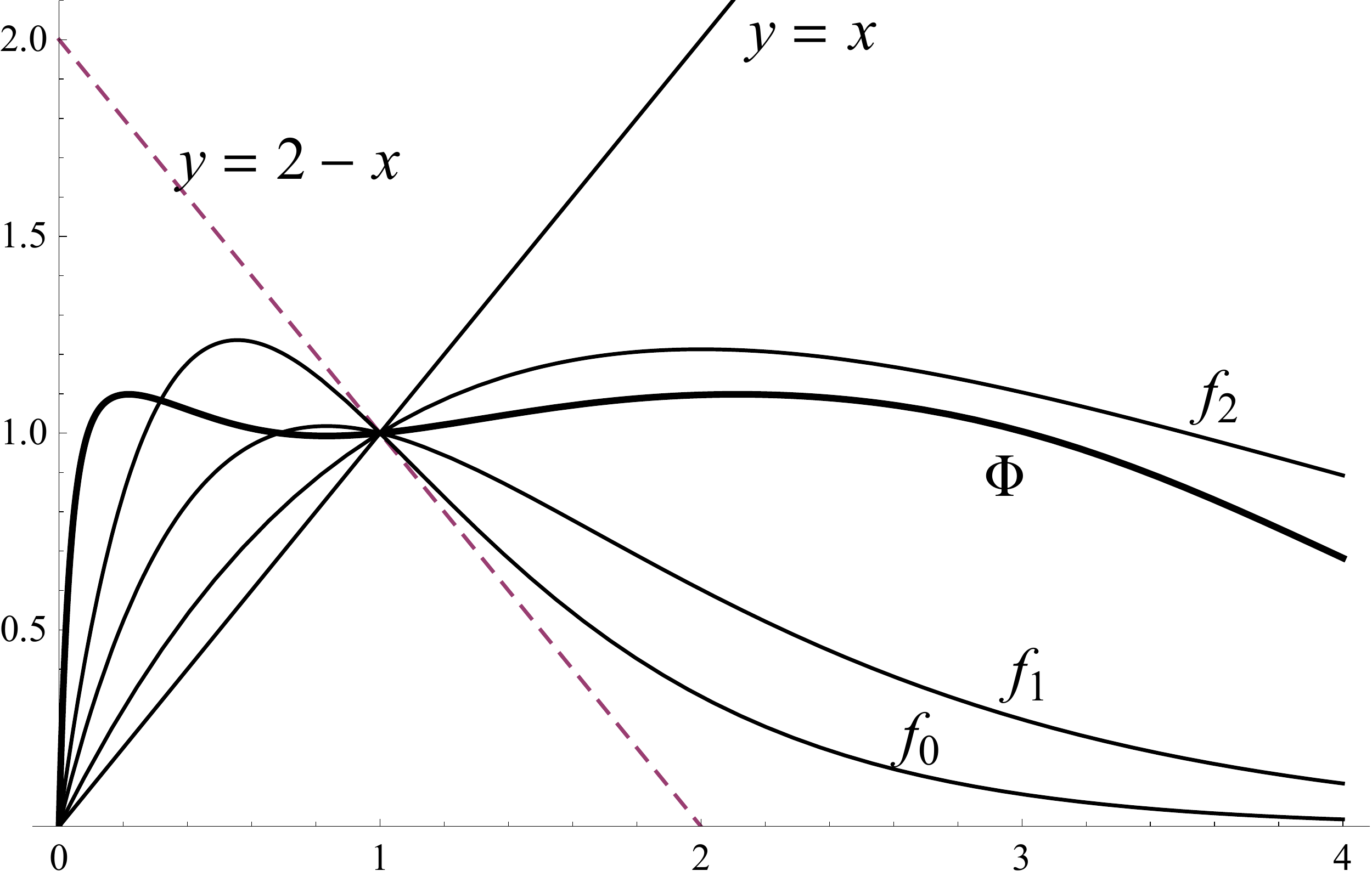}
\caption{An illustration of the ``individual enveloping" (tiny curves) and the ``composition enveloping" (solid curve) in the one-parameter family of a Ricker type map. In this case individual enveloping implies periodic  enveloping and consequently the global stability of the positive fixed point in the periodic equation.}
\label{fig:ERM}
\end{figure}

In Figure \ref{fig:ERM} is represented a concrete example where $r_0=1.8$, $r_1=1.2$ and $r_2=0.5$. The composition map $\Phi_3=R_2\circ R_1\circ R_0$ is represented by the solid curve. The dashed line is the enveloping function while the tiny curves are the individual population models.

Before ending this example, we notice that $\prod_{i=0}^{p-1}|1-r_i|<1$ is the stability condition  for the fixed point $x^*=1$ under the action of the composition map $\Phi_p$.   

Finally, we should mention that R. Sacker \cite{Sac2007} used a different method to show global stability of a similar periodic Ricker type model given by
\[
x_{n+1}=x_ne^{r_n-x_n}
\]
in the parameter region $0<r_n\leq 2$, $n=0,1,2,\dots,p-1$.

\subsection{Generalized Beverton-Holt model} 
Let $x_{n+1}=B_n(x_n)$, $n=0,1,2,\ldots$ where the map $B_n$ is given by
\begin{equation}
B_n(x)=\frac{\mu_nx}{1+(\mu_n-1)x^{c_n}}.
\label{beverton_map}
\end{equation}
Assume that $\mu_n>1$ and $0<c_n\leq 2$, for all $n=0,1,2,\dots$.   

The individual population map $B_n(x)$ has two fixed point, the origin and a positive fixed point given by $x^*=1$. It is easily shown that the origin is an unstable fixed point since $|B_n^\prime(0)|=\mu_n>1$, for all $n=0,1,2,\dots$. The condition of local stability of the positive fixed point is given by $\mu_n(c_n-2)\leq c_n$, for all $n=0,1,2,\dots$. This condition implies global stability  since each individual map $B_n(x)$, $n=0,1,2,\ldots$ is enveloped by  
\[
h(x)=\dfrac{1}{x},
\]
which is a decreasing fractional function with $h\circ h(x)=x$. Hence, \textbf{H2} is satisfied.

Let us assume now the periodicity of the map $B_n$ by taking $\mu_{n+q}=\mu_n$ and $c_{n+r}=c_n$ for some $q,r=1,2,3,\ldots$. This implies that $B_{n+p}=B_n$, where $p=lcm(q,r)$. From the fact that $r_n>1$, for all $n$, it follows that $1+(r_n-1)x^{c_n}>0$ whenever $x\in\mathbb{R}^+$. Hence, the composition of the Beverton-Holt models is well defined and thus we have the continuity of the composition. Consequently, \textbf{H1} is satisfied. It follows from Theorem \ref{GS} that 
\[
\Phi_p(x)=B_{p-1}\circ\ldots B_1\circ B_0(x)
\]
is a globally asymptotically stable population model. Consequently, 
the $p-$periodic Beverton-Holt equation $x_{n+1}=B_n(x_n)$, $B_{n+p}=B_n$, $n=0,1,2,\ldots$ is globally stable whenever $\mu_n>1$ and $0<c_n\leq 2$.

Notice that the condition of stability of the positive fixed point of  $\Phi_p$ 
is given by $$\prod_{i=0}^{p-1}|1+(\mu_i-1)(1-c_i)|<\prod_{i=0}^{p-1}\mu_i.$$ 

We should mention that when $c_n=1$, we have the classical Beverton-Holt model. In a series of papers \cite{CH2001,CH2002,SERS2004,ElSa2005,ElSa2005b,ElSa2005a,ESUAE,ES2006,Kocic2005,Kon2004,stevo2006}, the authors used a different method to study the global stability of the positive periodic cycle. 

In order to have a complete study of this model, it remains to study the cases where $c_n>2$ for all $n$, or a possible mixing case in the parameters $c_n$, i.e, some of the parameters are less or equal than $2$ and others are greater than $2$. Under these scenarios, the individual enveloping is
\[
h(x)=\left\{ 
\begin{array}{l}
\dfrac{1}{x}\text{ ~~if }c_{n}\leq 2 \\ \\
\dfrac{c_{n}-1-(c_{n}-2)x}{c_{n}-2-\left( c_{n}-3\right) x}\text{~~ if }c_{n}> 2
\end{array}
\right. .
\]
We point out that, in certain cases is possible to find a common enveloping for all the $p$ individual maps and hence global stability of the periodic equation. However, this is not the general case, as we show in the following concrete example.

Let $\mu_0=1.1$, $\mu_1=7$, $c_0=7.5$ and $c_1=2.3$. Hence, the individual maps are given by 
\[
f_0(x)=\dfrac{1.1x}{1+0.1x^{0.5}}~~\text{ and }~~f_1(x)=\dfrac{7x}{1+6x^{2.3}}.
\]
Both $f_0$ and $f_1$ are globally asymptotically stable population models with respect to the positive fixed point since they are enveloped by
\[
h_0(x)=\dfrac{6.5-5.5x}{5.5-4.5x}~~\text{ and }h_1(x)=\dfrac{1.3-0.3x}{0.3+0.7x},
\]
respectively (see Figure \ref{fig:No_Enveloping_Beverton_Holt}). 
\begin{figure}
\centering
\includegraphics[scale=0.5]{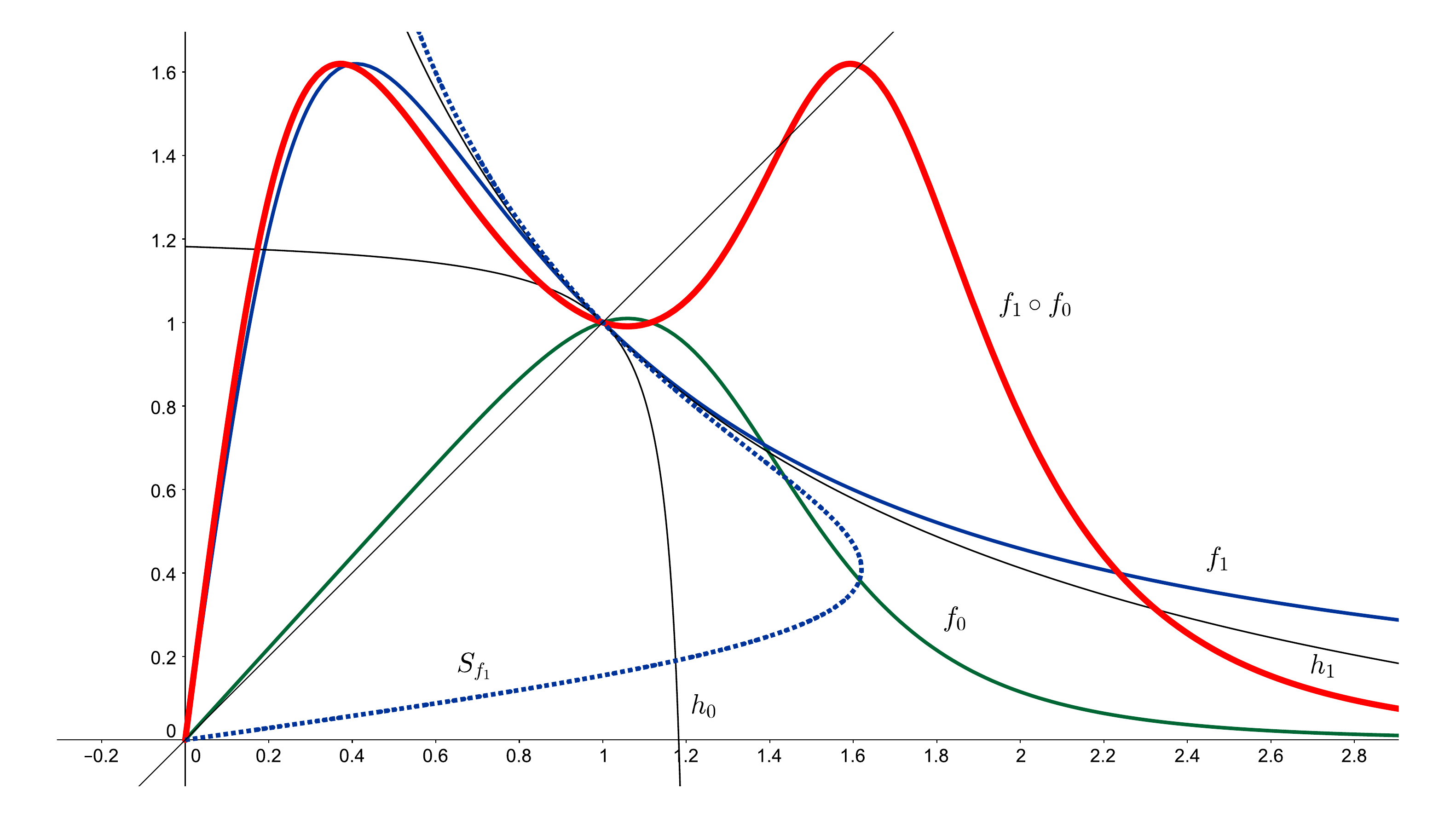}
\caption{An example showing that there is no Global Stability in the $2-$periodic Beverton-Holt model, when $\mu_0=1.1$, $\mu_1=7$, $c_0=7.5$ and $c_1=2.3$. In this case is not possible to find a decreasing enveloping that envelops simultaneously the individual population models $f_0$ and $f_1$.}
\label{fig:No_Enveloping_Beverton_Holt}
\end{figure}

As is clearly shown in Figure \ref{fig:No_Enveloping_Beverton_Holt} the composition map $\Phi_2(x)=f_1\circ f_0(x)$ has 3 positive fixed points and consequently can not be globally stable. Clearly, $\Phi_2$ is not a population model. From Corollary \ref{CGS} follows that  it can not exist a decreasing enveloping $h$, with $h\circ h(x)=x$, such that $h$ envelops simultaneously $f_0$ and $f_1$. 

Notice that, if that enveloping exits, then from Proposition  \ref{PENV} it lies between the graphs of $f_0$ and $f_1$ and the respective curves, $S_{f_0}$ and $S_{f_1}$, obtained from $f_0$ and $f_1$ by symmetry with respect to the diagonal $y=x$. As it is clearly shown in Figure \ref{fig:No_Enveloping_Beverton_Holt}, there exits an interval $]a,b[\subset]1.5,1.6[$ where such enveloping fails.

\subsection{Mixing models: Beverton-Holt acting with Ricker model}
Let us now consider that the sequence of maps is given by
\[
f_{n}(x)=\left\{ 
\begin{array}{l}
xe^{r_{n}(1-x)}\text{ if }n\text{ is even} \\ [10pt]

\dfrac{\mu _{n}x}{1+(\mu _{n}-1)x^{c_n}}\text{ if }n\text{ is odd}
\end{array}
\right. ,
\]
where $0<r_{n}\leq 2$, $\mu_{n}>1$ and $0<c_{n}\leq 1$ for all $n$. Assume the periodicity of the parameters, i.e., $r_{n+q}=r_n$, $\mu_{n+r}=\mu_n$ and $c_{n+s}=c_n$ for some positive integer $q$, $r$ and $s$. Hence, the equation $x_{n+1}=f_{n}(x_n)$ is $p-$periodic with $p=lcm (q,r,s)$. Clearly $f_n(x)\in \mathcal{F}$ and consequently \textbf{H1} is satisfied.

 Now, from Subsection \ref{grm} the sequence of maps $f_{2n}(x)$, $n=0,1,\ldots$ is enveloped by the decreasing function $h(x)=2-x$. It is easy to see that $$f_{2n+1}(x)=\dfrac{\mu _{2n+1}x}{1+(\mu _{2n+1}-1)x^{c_{2n+1}}},$$ $\mu_n>1$ and $0<c_n\leq 1$, $n=1,3,5,\ldots$ is also enveloped by $h(x)=2-x$. To see this observe that $f_{2n+1}(1)=h(1)$, $h$ is decreasing and $f_{2n+1}(x)$ is increasing since
\[
f^\prime_{2n+1}(x)=\dfrac{\mu_{2n+1}+\mu_{2n+1}(\mu_{2n+1}-1)(1-c_{2n})x^{c_{2n+1}}}{(1+(\mu _{2n+1}-1)x^{c_{2n+1}})^2}>0.
\]
This implies that \textbf{H2} is satisfied. Consequently,  from Theorem \ref{GS} follows that 
\[
\Phi_p(x)=f_{p-1}\circ\ldots f_1\circ f_0(x)
\]
is a globally asymptotically stable population model. Hence, 
the $p-$periodic difference equation $x_{n+1}=f_n(x_n)$, $n=0,1,2,\ldots$ is globally stable.

Finally we determine the stability condition under composition operator, which is given by $|\Phi^\prime_p(1)|<1$. A forward computation shows that
\[
\Phi ^{\prime }_p(1)=\left\{ 
\begin{array}{l}
\prod_{n=0}^{\frac{p}{2}-1}\dfrac{(1-r_{2n})(\mu_{2n+1}-(\mu_{2n+1}-1)c_{2n+1})}{\mu_{2n+1}}\text{   if }p\text{ is even} \\ [10pt]
(1-r_{p-1})\prod_{n=0}^{\frac{p-1}{2}-1}\dfrac{(1-r_{2n})(\mu_{2n+1}-(\mu_{2n+1}-1)c_{2n+1})}{\mu_{2n+1}}\text{ if }p\text{ is odd}%
\end{array}%
\right. .
\]

Notice that a similar approach can be done in the case that we consider the even sequence of maps a Beverton-Holt type and the odd sequence of maps a Ricker-type model.
\subsection{Exponential and rational}
Let us consider the non-autonomous difference equation
\[
x_{n+1}=\dfrac{(1+a_ne^{b_n})x_n}{1+a_ne^{b_nx_n}},
\]
where $0<b_n\leq 2$ and $a_n>0$, for all $n=0,1,2,\ldots$. This equation can be defined by the following map
\[
f_n(x)=\dfrac{(1+a_ne^{b_n})x}{1+a_ne^{b_nx}}.
\]
It is easy to check that the conditions of local stability at the fixed point $x^*=1$ $$a_n(b_n-2)e^{b_n}\leq 2$$ implies global stability since each map $f_n$ is enveloped by $h(x)=2-x$ in $\mathbb{R}^+_0$.

Let $a_{n+q}=a_n$ and $b_{n+r}=b_n$, for all $n=0,1,2,\ldots$. Then the sequence of maps is $p-$periodic where $p=lcm(q,r)$, i.e., the non-autonomous equation $x_{n+1}=f_{n}(x_n)$ is $p-$periodic. Clearly, the periodic composition map $\Phi_p(x)$ is continuous in $\mathbb{R}^+_0$ since $1+a_ne^{b_nx}\neq 0$ for all $x\geq 0$ and $0<b_n\leq 2$ and $a_n>0$, for all $n=0,1,2,\ldots$.

Since \textbf{H1} and \textbf{H2} are satisfied, from Theorem \ref{GS} follows the global stability in the non-autonomous periodic equation.

\subsection{Quadratic model}
Let $x_{n+1}=L_n(x_n)$, where $L_n(x)=x(1+\mu_n(1-x))$, $x\in I_n=\left[0,1+\dfrac{1}{\mu_n}\right]$, for all $n=0,1,2,\ldots$. The local stability condition for each individual population model $L_n(x)$, $n\in\{0,1,\ldots,p-1\}$ is given by $0<\mu_i\leq  2$. Since the fractional function $h(x)=\frac{4-3x}{3-2x}$ envelops each map $L_n(x)$, $n\in\{0,1,\ldots,p-1\}$, it follows that $L_n(x)$ is a globally asymptotically stable population model whenever $0<\mu_i\leq  2$. Notice that $h\circ h(x)=x$. Hence \textbf{H2} is satisfied.

Let us now assume the periodicity of the difference equation by taking $L_{n+p}=L_n$, for all $n=0,1,2,\ldots$, i.e., the sequence of  parameters are  $p-$periodic. In order to guarantee the continuity of the composition operator we construct the interval $J$ as follows. Let $I$ to be an interval given by
\[
I=\bigcap_{n=0}^{p-1}\left[0,1+\dfrac{1}{\mu_n}\right].
\]
The interval $J$ is defined by
\[
J=\bigcap_{i=0}^{p-1}L_n(I).
\]
Clearly $J\supseteq [0,1]$, $\Phi_p(1)=1$ and $\Phi_p(x)$ is continuous for all $x\in J$. Hence \textbf{H1}  is  satisfied. It follows from Theorem \ref{GS} that $$\Phi_p(x)=L_{p-1}\circ\ldots\circ L_1\circ L_0(x),~~x\in J $$ is a globally asymptotically stable population model whenever $0<\mu_n\leq  2$, $n=0,1,2,\ldots$. Consequently, $x^*=1$ is a globally stable fixed point of the  $p-$periodic difference equation  $x_{n+1}=L_n(x_n)$. 


\subsection{Beverton-Holt with harvesting}

Let us consider the difference equation $x_{n+1}=f_n(x_n)$, where the sequence of maps $f_n$ is given by
\[
f_n(x)=\dfrac{r_nx}{1+(r_n-1)x}-c_n x(x-1), r_n>1,~~0<c_n<1, \text{ for all }n.
\]
In this model we are taking $x$ in the interval $$I_n=\left[0, \dfrac{(r_n-2)\sqrt{c_n}-\sqrt{r_n(r_n(4+c_n)-4)}}{2(r_n-1)\sqrt{c_n}}\right].$$

Clearly $f_n(x)$ is a population model for all $x\in I_n$. Moreover, the extinction fixed point is unstable since $f_n^\prime(0)=r_n+c_n>1$ for all $n=0,1,2,\ldots$. The local stability condition for $x^*=1$ is given by
\[
0<c_n<\dfrac{1+r_n}{r_n}, \text{ for all }n.
\]
Since the fractional function $h(x)=\dfrac{11-8x}{8-5x}$ envelops each individual map $f_n(x)$, $n=0,1,2\dots$, it follows that $f_n(x)$ is a globally asymptotically stable population model whenever $x$ belongs to the interior of $I_n$ and $0<c_n<\dfrac{1+r_n}{r_n}$, $n=0,1,2\dots$. Notice that $h\circ h(x)=x$.

Let us now assume the periodicity of the parameters by taking $r_{n+q}=r_n$, and $c_{n+r}=c_n$ for all $n=0,1,2,\ldots$. Hence, the sequence of maps is $p-$periodic where $p=lcm(q,r)$. Define the interval $J$ as follows
\[
J=\bigcap_{n=0}^{p-1}f_n(I),~~\text{ where }~~ I=\bigcap_{n=0}^{p-1}I_n.
\]
Since \textbf{H1} and \textbf{H2} are satisfied, it follows from Theorem \ref{GS} that $$\Phi_p(x)=f_{p-1}\circ\ldots\circ f_1\circ f_0(x),~~x\in J$$ is a globally asymptotically stable population model
whenever $0<c_n<\dfrac{1+r_n}{r_n}$, for all $n=0,1,2,\ldots$, i.e., $x^*=1$ is a globally stable fixed point of the  $p-$periodic Beverton-Holt equation with harvesting.

\section{Final remarks}

In this paper, we have established the conditions  when individual enveloping implies periodic enveloping in certain  periodic population models. In other words, if a sequence of population models is enveloped by a common decreasing function $h$ such that $h(h(x))=x$, then the periodic equation is globally stable with respect to the positive equilibrium. 

In population dynamics, this observation stats that, if each one of the individual population is globally stable with respect to the positive equilibrium, then the population with fluctuation habitat is also globally stable with respect to the positive equilibrium.  A several examples are given in order to illustrate the results.

Hence, in the case of certain periodic forced systems, local stability implies global stability with respect to the positive fixed point.


\end{document}